\documentclass[onecolumn,notitlepage,11pt]{article}
\usepackage{sectsty}
\usepackage{mathrsfs}
\usepackage{latexsym}
\usepackage{amsbsy}
\usepackage{amssymb}
\usepackage{verbatim}
\usepackage{amsthm}
\usepackage{amsfonts}
\usepackage{amsmath}
\usepackage{graphicx}
\usepackage{layout}
\usepackage{pb-diagram}
\usepackage{amscd}
\usepackage{anysize}
\usepackage{tikz}
\usepackage{graphicx}
\usepackage[all,cmtip]{xy}
\usepackage{youngtab}
\usepackage{hyperref}

\newcommand{\Exp}{\mbox{Exp}}

\newcommand{\calA}{\mathcal{A}}
\newcommand{\calT}{\mathcal{T}}
\newcommand{\calE}{\mathcal{E}}
\newcommand{\calB}{\mathcal{B}}
\newcommand{\calS}{\mathcal{S}}
\newcommand{\calF}{\mathcal{F}}

\newcommand{\Q}{\mathbb{Q}}
\newcommand{\R}{\mathbb{R}}

\newcommand{\D}{D}
\newcommand{\On}{\mbox{O}}

\newcommand{\frakF}{\mathfrak{F}}
\newcommand{\metneg}{\mbox{MET}^{<0}}
\newcommand{\metnc}{\mbox{MET}^{\mathcal{NC}}}

\newcommand{\sDiff}{\mathbf{DIFF}}
\newcommand{\Diff}{\mbox{Diff}}
\newcommand{\Top}{\mbox{Top}}
\newcommand{\sTop}{\mathbf{TOP}}

\newcommand{\TDiffn}[2]{\frac{\Top_0(#1)}{\Diff_0(#2)}}

\newcommand{\sTDiffn}[2]{\frac{\sTop_0(#1)}{\sDiff_0(#2)}}

\newcommand{\TO}[2]{\frac{\Top(#1)}{\mbox{O}(#2)}}
\newcommand{\beq}{\begin{equation*}}
\newcommand{\eeq}{\end{equation*}}
\theoremstyle{definition}
\newtheorem{theorem}{Theorem}
\newtheorem*{theorem*}{Theorem}
\newtheorem{lemma}[theorem]{Lemma}

\newtheorem{propo}[theorem]{Proposition}
\newtheorem{coro}{Corollary}[theorem]
\newtheorem{rmk}{Remark}

\marginsize{2.5cm}{2.5cm}{2cm}{2cm} 
\title{\Large\bf Negatively curved bundles in the Igusa stable
range
}
\author{Mauricio Bustamante \and Francis Thomas Farrell \and Yi Jiang}
\newcommand{\Addresses}{{
  \bigskip
  \footnotesize
\noindent\textsc{Mauricio Bustamante}\par\nopagebreak\textsc{Institut f\"ur Mathematik, Universit\"at Augsburg}\par\nopagebreak
 \texttt{EMAIL: Mauricio.BustamanteLondono@math.uni-augsburg.de}\\
\textsc{Francis Thomas Farrell}\par\nopagebreak\textsc{Yau Mathematical Sciences Center, Tsinghua University, Beijing, China}\par\nopagebreak
 \texttt{EMAIL: farrell@math.tsinghua.edu.cn}\\
\textsc{Yi Jiang}\par\nopagebreak\textsc{Yau Mathematical Sciences Center, Tsinghua University, Beijing, China}\par\nopagebreak
 \texttt{EMAIL: yjiang117@mail.tsinghua.edu.cn}
 }}
\date{}
\begin{document}
\maketitle
\begin{abstract}
We use classical results in 
smoothing theory to extract information about
the rational homotopy groups of the space of  Riemannian metrics without conjugate points
on a high dimensional manifold with hyperbolic fundamental group. As a consequence, we show that spaces of negatively curved Riemannian metrics have in general  nontrivial rational homotopy groups. We also show that
smooth $M$-bundles over spheres equipped with fiberwise negatively 
curved metrics 
represent elements of finite order in
the homotopy groups $\pi_i B\Diff(M)$ of the classifying space 
for smooth $M$-bundles, provided $i\ll\mbox{dim } M$.
\end{abstract}
\section{Introduction}
Let $M$ be a closed smooth manifold. A negatively curved bundle with fiber $M$ is a smooth $M$-bundle 
$E\to B$ whose fibers are endowed with continuously varying 
Riemannian metrics of everywhere negative sectional curvature.
This notion has been established in $\cite{FOGAFA}$, where a
theory for negatively curved fiber bundles is developed, in the
sense that there is a space $\calT^{<0}(M)$ with the property
that equivalence classes of fiber-homotopically trivial negatively
curved bundles over a paracompact space $B$ are in bijective correspondence with
homotopy classes of maps $B\to\calT^{<0}(M)$. Here two negatively
curved $M$-bundles $E_1\to B$ and $E_2\to B$ are equivalent if
there exists a negatively curved $M$-bundle $\mathcal{E}$ over
$B\times [0,1]$ such that $\mathcal{E}$
restricted to $B\times\{i\}$ is fiberwise isometric to $E_i$,
$i=1, 2$ (see \cite[p.1399]{FOGAFA}).
From a purely topological point of view, fiber homotopically trivial
smooth bundles over $B$ with fiber a negatively curved manifold
$M$ are classified, up to bundle equivalence, by homotopy classes of
maps $B\to B\Diff_0(M)$, where $\Diff_0(M)$ is the group of all
diffeomorphisms of $M$ which are homotopic to the identity on $M$.
Thus there is a natural ``forgetful'' map
\beq
\frakF:\calT^{<0}(M)\to B\Diff_0(M).
\eeq
One then wonders how much these two bundle theories differ, and this is
the theme of this paper. 
This question has been addressed
already by Farrell and Ontaneda in \cite{FOAnnals,FO10,FOGAFA}. A remarkable
observation is that the homotopy fiber of the forgetful map 
$\frakF$ can be identified with the space $\metneg(M)$ of 
all negatively
curved metrics on $M$, so that we have a 
homotopy fibration
\begin{equation}\label{fibration}
\metneg(M)\to\calT^{<0}(M)\xrightarrow{\frakF} B\Diff_0(M).
\end{equation}
Farrell and Ontaneda \cite{FO10} have shown that the space of 
negatively curved metrics on $M$ is highly
non-connected if the dimension of $M$ is sufficiently large, thus
the two bundle theories are fundamentally different. This difference was captured in their main theorem by elements of finite order
in the homotopy groups of $\metneg(M)$. 
Nevertheless, one could still hope that the two theories were 
``rationally equivalent'' (this kind of phenomenon occurs
for example in
the theory of stable vector
bundles and stable topological $\R^n$-bundles. Indeed, $B\On$ and
$B\Top$ are only rationally equivalent). 
Our main result establishes that these two bundle theories
are inequivalent even if one decides to neglect torsion,
at least in 
a range of dimensions called the Igusa stable range. 
This is done by showing that the rational homotopy groups of 
$\metneg(M)$ are in general nontrivial. 

Let us now assume that $M$ is a closed manifold that
supports a Riemannian metric $g$ \textit{without conjugate points}, that is 
for any geodesic $\gamma$ of $M$ no two points are conjugate along $\gamma$
(e.g. nonpositively curved metrics, or metrics with geodesic
flow of Anosov type). In particular $M$ is an aspherical manifold and for any point $p\in M$ the exponential map ${\rm{exp}}_p:T_pM\to M$ is a (universal) covering map. 
Let $\metnc(M)$ denote the space of all such Riemannian metrics on $M$.
Then given any self-diffeomorphism $f:M\to M$ of $M$,
we define the push-forward metric $g'$ on $M$ to be 
the unique Riemannian metric such that 
$f:(M,g)\to (M,g')$ is an isometry. We denote the
push-forward metric by $f_*g$. This gives rise to an orbit
map $\Phi^g:\Diff_0(M)\to\metnc(M)$, defined by
$\Phi^g(f)=f_*g$, and the corresponding map in homotopy groups
\beq
 \Phi^g_*:\pi_i\Diff_0(M)\to\pi_i\metnc(M).
\eeq

\begin{theorem}\label{main}
Let $(M,g)$ be a closed Riemannian $n$-manifold without conjugate 
points. Assume that the fundamental group of $M$ is hyperbolic.
Then for all $1<i<\min\{\frac{n-10}{2},\frac{n-8}{3}\}$ the  map
\beq
\Phi_*^g:\pi_i(\Diff_0(M),id)\otimes\Q\to\pi_i(\metnc(M),g)\otimes\Q
\eeq
is injective.
\end{theorem}

For the definition of a \textit{hyperbolic group} we refer the reader to 
\cite{bridson} or \cite{gromov}.

Note that Theorem \ref{main} does give
elements of infinite order in $\pi_i\metnc(M)$ in 
view of the following
result of Farrell and Hsiang. First recall that a
closed manifold $M$ \textit{satisfies 
the strong Borel conjecture} if for all $k\geq 0$,
every self-homotopy equivalence of pairs 
$(M\times\D^k , M \times S^{k-1})\to (M \times\D^k , M\times S^{k-1})$
which is a homeomorphism when restricted to the boundary
$M\times S^{k-1}$ is homotopic (relative to the
boundary) to a homeomorphism.

\begin{theorem*}[Farrell-Hsiang \cite{FH78}]
Let $M$ be a closed aspherical smooth $n$-manifold that satisfies both the
strong Borel conjecture and Conjecture 2 in \cite[p.326]{FH78}.
Then if $0<i<\min\{\frac{n-7}{2},\frac{n-4}{3}\}$\footnote{By the time Farrell and Hsiang obtained this result the stability range
was roughly $i<n/6$. Later Igusa improved this range to the one stated here. See \cite{Igusa-stability} and \cite[p.252]{Igusa-book}.},
\begin{equation}\label{calculation}
\pi_i\Diff_0(M)\otimes\Q=
\begin{cases}
\displaystyle\bigoplus_{j=1}^{\infty}H_{(i+1)-4j}(M,\Q)&\text{ if } n\text{ is odd}\\
0&\text{ if } n\text{ is even}
\end{cases}
\end{equation} 
\end{theorem*}

Some important classes of  manifolds relevant to the subject of this paper which satisfy both the strong Borel
conjecture and Conjecture 2 in \cite[p.326]{FH78} are: 
closed negatively curved manifolds \cite{FJ89} or more generally 
closed nonpositively 
curved manifolds \cite{FJ93} and closed aspherical
manifolds with hyperbolic fundamental group \cite{BLR}, \cite{Bartels}. 
For all those, the 
calculation (\ref{calculation}) holds.

In the case of a negatively curved Riemannian manifold $(M,g)$ 
the orbit map $\Phi^g$ factors through the inclusion 
\beq
\metneg(M)\hookrightarrow\metnc(M),
\eeq
and since the fundamental group of $M$ is hyperbolic, we have:
\begin{coro}\label{coroforneg}
Let $(M,g)$ be a closed negatively curved Riemannian $n$-manifold.
Then for all $1<i<\min\{\frac{n-10}{2},\frac{n-8}{3}\}$ the  map
\beq
\Phi_*^g:\pi_i(\Diff_0(M),id)\otimes\Q\to\pi_i(\metneg(M),g)\otimes\Q
\eeq
is injective. 
\end{coro}

\begin{rmk}
Similar corollaries can be obtained in the same way 
for various spaces of Riemannian 
metrics. Some other examples are: 
the space $\mbox{MET}^{\leq 0}(M)$ of nonpositively 
curved Riemannian metrics on a closed negatively curved manifold $M$
(considered already in \cite{FO15}), and the space
$\mbox{MET}^{\calA}(M)$ of Riemannian metrics with geodesic flow of Anosov
type (see \cite{klin} for the non-conjugate points condition and \cite{ruggiero-asterisque, ruggiero-expansive} for hyperbolicity
of the fundamental group). 
\end{rmk}

Corollary \ref{coroforneg}
together with the long exact sequence
in homotopy groups
for the fibration (\ref{fibration}) imply that,
in the Igusa stable 
range, every negatively curved bundle over an $(i+1)$-sphere 
has finite order (viewed as an element in $\pi_{i+1}B\Diff_0(M)$).
More precisely:
\begin{coro}\label{coroneg}
Let $(M,g)$ be a closed negatively curved Riemannian $n$-manifold.
Then for all $2<i<\min\{\frac{n-10}{2},\frac{n-8}{3}\}$
there is a short exact sequence
\beq
0\to\pi_i\Diff_0(M)\otimes\Q\to\pi_i\metneg(M)\otimes\Q
\to\pi_i\calT^{<0}(M)\otimes\Q\to 0.
\eeq
In particular, the forgetful map 
$\mathfrak{F}\otimes id_{\Q}:
\pi_{i+1}\calT^{<0}(M)\otimes\Q\to \pi_{i+1}B\Diff_0(M)\otimes\Q$ 
is trivial for all $1<i<\min\{\frac{n-10}{2},\frac{n-8}{3}\}$.
\end{coro}

Instead of looking at smooth negatively 
curved Riemannian metrics,
one could consider the weaker notion of a $\mbox{CAT(-1)}$ metric
on the closed smooth manifold $M$, that
is metrics on $M$ with the property that every geodesic
triangle is ``thinner'' than the corresponding 
comparison triangle in the hyperbolic plane 
(for a precise definition see \cite{bridson}). 
Let $\mbox{CAT}^{-1}(M)$ be the space of all such metrics on $M$
with the $C^0$-topology. There is a continuous map 
$\beta:\mbox{MET}^{<-1}(M)\to\mbox{CAT}^{-1}(M)$
given by sending each negatively curved metric $M$ to the 
induced distance function on $M$. Here $\mbox{MET}^{<-1}(M)$ is the
space of all metrics on $M$ with curvature less than -1.
\begin{coro}\label{cat}
There exist an $i>0$, a closed negatively curved manifold $M$ and a 
non-trivial smooth $M$-bundle over the $(i+1)$-sphere $S^{i+1}$, which
admits a continuously varying family of $\mbox{CAT}(-1)$ metrics on the
fibers  and
it is \textit{not} a negatively curved bundle. 
\end{coro}
\begin{proof}
Take $(M^n,g)$ to be an odd-dimensional closed negatively curved manifold such
that $H_{(i+1)-4j}(M,\Q)\neq 0$ for some $1<i<\min\{\frac{n-10}{2},\frac{n-8}{3}\}$ and some $j\geq 1$ (e.g.
for $i=3$ and $j=1$,  any connected $19$-dimensional negatively curved manifold satisfies this condition).  By (\ref{calculation}), there
is a map $f:S^{i}\to\Diff_0(M)$ which represents a non-trivial class $[f]$ in
$\pi_i\Diff_0(M)\otimes\Q$, but maps to zero in $\pi_i\Top_0(M)$. This is possible since $\pi_i\Top_0(M)$
is a torsion abelian group when $1<i<\min\{\frac{n-7}{2},\frac{n-4}{3}\}$ by  \cite{FJ93} (see also \cite{F-ICTP}). Now
perform the the clutching construction, that is glue two 
copies of $D^{i+1}\times M$ with their boundaries
identified by $(z,y)\mapsto (z,f(z)(y))$, for 
$(z,y)\in \partial D^{i+1}\times M$. This gives rise to non-trivial
smooth $M$-bundle over the $(i+1)$-sphere.  This bundle is 
not a negatively curved bundle. Indeed, if it were, then $[f]$ would be in the
image of the forget map $\mathfrak{F}\otimes id_{\Q}:
\pi_{i+1}\calT^{<0}(M)\otimes\Q\to \pi_{i+1}B\Diff_0(M)\otimes\Q$ . But this  is not possible by 
Corollary \ref{coroneg}. To show that the bundle admits a continuously varying
family of $\mbox{CAT}(-1)$ metrics, it suffices to show that the map
$S^i \to \mbox{CAT}^{-1}(M)$ given by
$z\mapsto \beta(f(z)_*g)$ extends over the $(i+1)$-disc. 
To see this, note that $\Top_0(M)$ acts continuously 
on $\mbox{CAT}^{-1}(M)$ by pushing-forward the distance
function induced by $g$.
One can form the following commutative diagram
\beq
\xymatrix{
\pi_i\Diff_0(M)\otimes\Q\ar[r]\ar[d] & \pi_i\mbox{MET}^{<-1}(M)\otimes\Q\ar[d]^{\beta_*}\\
\pi_i\Top_0(M)\otimes\Q\ar[r]	& \pi_i\mbox{CAT}^{-1}(M)\otimes\Q
}
\eeq
As noted above, the lower left group vanishes for $1<i<\min\{\frac{n-7}{2},\frac{n-4}{3}\}$. 
This completes the proof of the corollary.
\end{proof}

%

\subsection*{Comments on other related work}
\begin{enumerate}
\item As mentioned before
Farrell and Ontaneda \cite{FO10} 
obtain torsion elements in the homotopy groups of the space
of negatively curved Riemannian metrics. Their method would
never give infinite order elements because it
depends strongly on the existence of non-trivial 
elements in the homotopy groups of the space of 
stable \textit{topological}  
pseudoisotopies of the circle. But in Igusa's stable range, those
are torsion groups \cite{waldhausen}. 
\item Farrell and Ontaneda \cite{FOAnnals} 
find non-trivial 
elements of finite order
in $\pi_i\calT^{<0}(M)$, when $i\ll\mbox{dim}\, M$, and $M$ is
a real hyperbolic manifold. But again their method fails to give
information about infinite order elements in these homotopy
groups. The reason here is that
their construction
relies on finding nonzero elements in
$\pi_i\Diff(\D^n,\partial)$ which 
have preimages under the Gromoll map
$\pi_{i+1}\Diff(\D^{n-1},\partial)\to\pi_{i}\Diff(\D^{n},\partial)$, where $\Diff(\D^n,\partial)$ is the group of diffeomorphisms of
a closed disc $\D^n$ which are the identity on the boundary. But in the Igusa stability range, 
this map is rationally trivial \cite{FH78}. Nonetheless
elements of infinite order in 
$\pi_i\calT^{<0}(M)$ can be obtained if we look outside Igusa's stable
range. We elaborate on this in \cite{BFJ-Teich}.
\end{enumerate}

Our main tool to prove Theorem \ref{main} is Morlet's comparison theorem, which
we recall in Section \ref{background}.
In Section \ref{outline} we reduce Theorem \ref{main}
to two lemmas: one of pure topological nature (proved in 
Section \ref{prooflemma2}), and one that incorporates the geometry of 
$(M,g)$ (proved in the last section).
\subsection*{Acknowledgements}
M. Bustamante is supported by the DFG Priority Programme \textit{Geometry at Infinity} (SPP 2026): 
\textit{Spaces and moduli spaces of Riemannian metrics with curvature bounds on compact and non-compact manifolds.}
 He acknowledges the hospitality of the Yau Mathematical Sciences Center at Tsinghua University in Beijing, 
 where he was a postdoc when most of the results of this paper were obtained.  
\section{Background and notation}\label{background}
Throughout this paper, $\Top(X)$ denotes the group of 
self-homeomorphisms of
a topological space $X$, endowed with the compact open topology. 
If $X$ is a 
closed smooth manifold then 
$\Diff(X)$ denotes
the group of smooth self-diffeomorphisms of $X$ 
with the smooth topology. 
We denote by $\Top_0(X)$
(resp. $\Diff_0(X)$) the subgroup of $\Top(X)$
(resp. $\Diff(X)$) consisting of all those self-homeomorphisms
(resp. self-diffeomorphisms) of $X$ which are \textit{homotopic}
to the identity map.
Also, $\Top(n)$ will denote the group of homeomorphisms of
$\R^n$ with the compact open topology. We keep the customary notation $\On(n)$ for the group
of orthogonal (w.r.t. Euclidean metric) transformations of $\R^n$ with its usual topology.
\subsection{Morlet's comparison theorem}\label{Morlet} 
Let $M$ be a closed smooth $n$-dimensional manifold.
Let $\sTop(M)$ and 
$\sDiff(M)$ denote the singular and
smooth singular complex of $\Top(M)$ and 
$\Diff(M)$ respectively.
These are
simplicial sets whose geometric realizations are weakly
homotopy equivalent to $\Top(M)$ and 
$\Diff(M)$ respectively. A $k$-simplex
of $\sTop(M)$ (resp. $\sDiff(M)$) 
is a homeomorphism
(resp. diffeomorphism)
$\Delta^k\times M\to\Delta^k\times M$ which commutes with the
projection to the standard $k$ simplex $\Delta^k$.
The corresponding singular and smooth singular complexes
of $\Top_0(M)$ and $\Diff_0(M)$ are denoted by 
$\sTop_0(M)$ and $\sDiff_0(M)$ respectively.

One
advantage of working in this setting is that the simplicial
group $\sDiff_0(M)$ acts freely on $\sTop_0(M)$
and so the quotient
$\sTDiffn{M}{M}$ is naturally
a simplicial set, in fact a 
Kan complex. Furthermore there
is a Kan fibration
\beq
\sDiff_0(M)\to\sTop_0(M)\to
\sTDiffn{M}{M},
\eeq
which gives rise to the following long exact sequence of 
homotopy groups
\beq
\cdots\to\pi_{i+1}\Top_0(M)\to\pi_{i+1}
\TDiffn{M}{M}\xrightarrow{d_*}\pi_i\Diff_0(M)\to\pi_i\Top_0(M)\to\cdots
\eeq
where we denote by $\TDiffn{M}{M}$ the geometric realization of
 $\sTDiffn{M}{M}$, and the homotopy groups of  $|\mathbf{TOP}_0|$ (resp. $|\mathbf{DIFF}_0|$)
are identified with those of $\Top_0(M)$ (resp. $\Diff_0(M)$). 

Assume now that $M$ comes equipped with a Riemannian 
metric, and regard its tangent bundle $TM$ as a Euclidean
vector bundle. There is an associated (right) 
principal $\On(n)$-bundle $PM\to M$. 

Note that
$\On(n)$ acts on the left on the coset space $\TO{n}{n}$. Thus
we can form the balanced product
\beq
\calB_n(M):=PM\times_{\On(n)}\TO{n}{n},
\eeq
which is a fiber bundle over $M$ with fiber $\TO{n}{n}$.

The space (with the
compact-open topology) of 
sections of $\calB_n(M)$ is denoted by 
$\Gamma\left(\calB_n(M)\right)$.
Note that $\Gamma\left(\calB_n(M)\right)$ 
has a preferred element $s_0:M\to\calB_n(M)$
given by the $\On(n)$-invariant point $id\ \On(n)\in\Top(n)/\On(n)$.

The space $\frac{\Top_0(M)}{\Diff_0(M)}$ can be related to the space
of sections $\Gamma(\calB_n(M))$. We briefly recall how this is done, following the ideas of \cite{BLTrans}. Consider simplicial sets 
$\overline{\textbf{R}}^T(M)\supset$ 
$\overline{\textbf{R}}^d(M)\supset$ 
$\overline{\textbf{R}}^o(M)$ 
of topological, 
linear and orthogonal representations 
of the tangent bundle of $M$ respectively. That is,
a $k$-simplex of
$\overline{\textbf{R}}^T(M)$ is a topological $\R^n$-bundle isomorphism
$\Delta^k\times TM\to\Delta^k\times TM$ which leaves invariant the image of the zero section, and which commutes with
the projection to $\Delta^k$ and
covers some homeomorphism 
$\Delta^k\times M\to\Delta^k\times M$. 
The $k$-simplices of
$\overline{\textbf{R}}^d(M)$ 
and $\overline{\textbf{R}}^o(M)$
are defined similarly but 
the maps $\Delta^k\times TM\to\Delta^k\times TM$ are now
vector 
bundle isomorphisms and vector bundle isomorphisms which are
fiberwise isometries, respectively. We also consider a simplicial set $\overline{\textbf{R}}^t(M)$
whose $k$-simplices are germs of topological microbundle isomorphisms of $\Delta^k\times TM$ which commute
with the projection onto $\Delta^k$, and cover some homeomorphism $\Delta^k\times M\to \Delta^k\times M$.

Simplicial sets $\underline{\textbf{R}}^t(M)$, $\underline{\textbf{R}}^T(M)$,
$\underline{\textbf{R}}^d(M)$ and $\underline{\textbf{R}}^o(M)$
are defined analogously, the
only difference is that the
isomorphisms $\Delta^k\times TM\to\Delta^k\times TM$ must cover the identity map on $\Delta^k\times M$.

Now recall that the tangent microbundle of $M$ 
is the diagram 
$M\xrightarrow{\Delta}M\times M\xrightarrow{pr_1}M$, where
$\Delta$ is the diagonal map and $pr_1$ is the projection onto the
first factor. Since  we have fixed a
Riemannian metric on $M$, the
exponential map gives an isomorphism between $TM$ and the tangent
microbundle of $M$ (see e.g. \cite[p.56-58]{milnor-micro}). 
Thus, via this isomorphism,
any self-homeomorphism $f:M\to M$ induces, by taking $f\times f$,
a microbundle automorphism of $TM$ which is called
the \textit{topological derivative} of $f$.
This gives rise to a
simplicial map (c.f. \cite[p.453]{BLAnnals})
\begin{equation}\label{derivative}
\delta:\sTDiffn{M}{M}\to
\overline{\textbf{R}}^t(M)/\overline{\textbf{R}}^d(M)
\end{equation}

Morlet's comparison theorem is the statement that the topological 
derivative $\delta$ induces an injective map on 
connected components and an isomorphism
on higher homotopy groups, provided $M$ is a closed manifold of
dimension $\neq 4$ \cite[Proposition 4.3]{BLTrans}.

To complete the picture, we relate  $\overline{\textbf{R}}^t(M)/\overline{\textbf{R}}^d(M)$
to the space of sections $\Gamma(\calB_n(M))$ (or rather its singular simplicial complex 
$\textbf{S}\Gamma(\calB_n(M)))$ via the following two lemmas, whose proofs are an elaboration of the argument given by Burghelea and Lashof in
 \cite[Theorem 4.2 ($1^t$)]{BLTrans}:
\begin{lemma}\label{R^t}
The simplicial sets
$\overline{\textbf{R}}^t(M)/\overline{\textbf{R}}^d(M)$
and
$\underline{\textbf{R}}^T(M)/\underline{\textbf{R}}^o(M)$
are weakly homotopy equivalent.
\end{lemma}

\begin{lemma}\label{the map S}
There exits a map 
$S:\underline{\textbf{R}}^T(M)/\underline{\textbf{R}}^o(M)
\rightarrow
\textbf{S}\Gamma(\calB_n(M))$ which
induces an injective map on connected components and a weak homotopy equivalence on any connected component.
\end{lemma}

The proofs of the lemmas are deferred to the Appendix below.

By putting together the maps  (\ref{derivative}) and (\ref{other maps}) (in the Appendix), the map $S$,
and taking homotopy groups we obtain a map
\beq
\mu_*:
\pi_{i}\TDiffn{M}{M}\to
\pi_{i}\Gamma\left(\calB_{n}(M)\right),
\eeq
which is injective for $i=0$ and an isomorphism for all $i>0$ for every component that is hit.
This is called Morlet's isomorphism.

\begin{rmk}
There is a version of Morlet's theorem for smooth compact
manifolds with
boundary (see \cite{BLTrans}). In particular it follows from
this and the Alexander trick, that
there is a weak homotopy equivalence 
$\Diff(\D^n,\partial)\to\displaystyle\Omega^{n+1}\TO{n}{n}$, when
$n\neq 4$.
\end{rmk}
\section{Outline of the proof of Theorem \ref{main}}\label{outline}
Let $\D^n$ denote a closed unit disc in $\R^n$ centered at the
the origin and let $\mbox{int}\D^n$ be its interior. 
Observe that the balanced product
\beq
\calB_{n}^{\D}(M)=PM\times_{\On(n)}\TO{\D^n}{n}
\eeq
is a fiber bundle over $M$ with fiber $\TO{\D^n}{n}$ which has
a preferred section, namely the obvious map 
$s_1: M=PM\times_{\On(n)}\frac{\On(n)}{\On(n)}\to PM\times_{\On(n)}\TO{\D^n}{n}$. 
\begin{lemma}\label{Lemma 2}
Let $M$ be a closed smooth $n$-dimensional manifold with
$n$ odd. Then 
for all
$1<i<\min\{\frac{n-8}{2},\frac{n-5}{3}\}$ 
\beq
\pi_{i}\left(\Gamma\left(\calB_{n}^{\D}(M)\right),s_1\right)\otimes\Q=0.
\eeq
\end{lemma}
Throughout 
we fix a homeomorphism $\varphi:\mbox{int }\D^n\to \R^n$ defined
by 
\beq
\varphi(v)=\frac{v}{1-||v||}.
\eeq
Conjugation by $\varphi$ defines a continuous injective map 
$\mbox{int}_{\varphi}:\Top(\D^n)\to\Top(n)$, that is
\beq
\mbox{int}_{\varphi}(f)=\varphi\circ f|_{\mbox{int}(\D^n)}\circ\varphi^{-1}
\eeq
Clearly this map is equivariant with respect to both the right and the left
action of 
$\On(n)$, and hence it induces a left equivariant map between orbit spaces that
we keep denoting by 
$\mbox{int}_{\varphi}:\TO{\D^n}{n}\to\TO{n}{n}$. 
Furthermore
$\mbox{int}_{\varphi}$ gives rise to a map between spaces of
sections
\beq
\mbox{int}_{\varphi}:
\left(\Gamma\left(\calB_n^{\D}(M)\right),s_1\right)\to
\left(\Gamma\left(\calB_n(M)\right),s_0\right).
\eeq
\begin{lemma}\label{Lemma 1}
Let $g\in\metnc(M)$. Assume that $\pi_1M$ is a hyperbolic group.
Then for all $i \geq 0$ and
$y\in 
\ker\Phi^g_*\circ d_*\subset\pi_{i+1}\TDiffn{M}{M}$, there exists 
$\bar{y}\in\pi_{i+1}\Gamma\left(\calB_{n}^{\D}(M)\right)$ such that
\beq
\mbox{int}_{\varphi\ast}(\bar{y})=\mu_*(y)\in\pi_{i+1}\Gamma(\calB_n(M))
\eeq
\end{lemma}
\begin{proof}[\textbf{Proof of Theorem \ref{main} assuming Lemma \ref{Lemma 2} and Lemma \ref{Lemma 1}}]
If $n$ is even the orbit map is obviously rationally 
injective by (\ref{calculation}).
If $n$ is odd then Lemma \ref{Lemma 2} and Lemma \ref{Lemma 1}
imply that $\ker\left(\Phi^g_*\circ d_*\right)\otimes id_{\Q}=0$ for $1<i<\min\{\frac{n-10}{2},\frac{n-8}{3}\}$.
But it is known (see e.g. \cite{F-ICTP}, \cite{FJ93}) that
$\pi_i\Top_0(M)\otimes\Q=0$ provided
$M$ is aspherical and satisfies the strong Borel conjecture
and $1<i<\min\{\frac{n-7}{2},
\frac{n-4}{3}\}$. That manifolds with hyperbolic
fundamental group satisfy both the strong Borel conjecture and Conjecture 2 in \cite[p.326]{FH78}, follows from \cite{BLR} and \cite{Bartels}.
Hence 
$d_*\otimes id_{\Q}:\pi_{i+1}\frac{\Top_0(M)}{\Diff_0(M)}\otimes\Q\to\pi_i\Diff_0(M)\otimes\Q$ is an
isomorphism when $1<i<\min\{\frac{n-9}{2},\frac{n-7}{3}\}$. Therefore the orbit map 
$\Phi_*^g\otimes id_{\Q}:\pi_i\Diff_0(M)\otimes\Q
\to\pi_i\metnc(M)\otimes\Q$ must be injective for $1<i<\min\{\frac{n-10}{2},\frac{n-8}{3}\}$.
\end{proof}
The rest of the paper is devoted to the proof of Lemma \ref{Lemma 2}
and Lemma \ref{Lemma 1}.

\section{Proof of Lemma \ref{Lemma 2}}\label{prooflemma2}
The rationalization of a nilpotent space $X$ will be
denoted by $X_{(0)}$. A map $f:X\to Y$ 
between nilpotent spaces is 
\textit{rationally $k$-connected} if the induced map 
$f_{(0)}:X_{(0)}\to Y_{(0)}$ between their rationalizations is
$k$-connected.

We have the following general fact (compare \cite[Lemma 3.10]{BLAnnals}).
\begin{lemma}\label{general}
Let $E_i\to B$, $i=1,2$ be two fibrations over a finite
CW-complex $B$ whose fibers $F_i$, $i=1,2$,
over a point $\ast\in B$ are simply connected
spaces.
Let $\beta:E_1\to E_2$ be a continuous map over the identity on
$B$ such that the restriction $\beta:F_1\to F_2$ is rationally $k$-connected. 
Then the induced map $\Gamma(\beta):\Gamma(E_1)\to\Gamma (E_2)$ between
section spaces is rationally $(k-\mbox{dim } B)$-connected.
In particular,  if $\pi_{j}F_1\otimes\Q=0$ for $1< j\leq\ell$,
then $\pi_j\Gamma(E_1)\otimes\Q=0$ for $1< j\leq\ell-\mbox{dim }B$.
\end{lemma}
\begin{proof}
Assume first that the fibers $F_i$, $i=1, 2$ have the 
homotopy type of a CW-complex. Fiberwise rationalize 
the two fibrations
(see \cite[p.40]{BousfieldKan} or \cite{Llerena})
to obtain another pair of fibrations
$F_{i(0)}\to \mathcal{E}_i\to B$ over $B$ and a fiber preserving
map $\mathcal{E}_{1}\to\mathcal{E}_{2}$. Since the restriction 
$\beta_{(0)}:F_{1(0)}\to F_{2(0)}$ of this
map is assumed to be $k$-connected, we have that
the induced map 
$\Gamma(\mathcal{E}_1)\to\Gamma(\mathcal{E}_2)$ 
between spaces of sections is
$(k-\mbox{dim }B)$-connected (\cite[Lemma 3.10]{BLAnnals}).
The lemma now follows in this case since $\Gamma(\mathcal{E}_i)$,
$i=1,2$ is nilpotent and fiberwise localization
commutes with taking spaces of sections, as shown by M\o ller in
\cite[Theorem 5.3]{Moller}.

In order to handle the case when the fibers $F_i$ 
may not have the homotopy type of a CW-complex, let $\textbf{S}$ 
be the functor that assigns
to a space its singular complex. 
We let $\mathcal{G}_i\to B$ be the
pull-back of the fibration 
$|\textbf{S}E_i|\to |\textbf{S}B|$ along a fixed homotopy
inverse $\varphi$ of the natural map 
$\psi:|\textbf{S}B|\to B$. Then there is a commutative
diagram
\beq
\xymatrix{
\mathcal{G}_1\ar[r]^{\overline{\beta}}\ar[d] & \mathcal{G}_2\ar[d]\\
\varphi^*\psi^*E_1\ar[r]\ar[d] & \varphi^*\psi^*E_2\ar[d]\\
E_1\ar[r]^{\beta} &E_2,
}
\eeq
where the vertical maps cover the identity map on $B$ and restrict
to weakly homotopy equivalences on the fibers. Since the
restriction $\beta:F_1\to F_2$ is rationally $k$-connected,
so is the restriction of $\overline{\beta}$ to the fiber, and hence
the induced map $\Gamma(\overline{\beta}):
\Gamma(\mathcal{G}_1)\to\Gamma(\mathcal{G}_2)$ is rationally
$(k-\mbox{dim }B)$-connected by the previous case. Consequently 
$\Gamma(\beta)$ is rationally $(k-\mbox{dim }B)$-connected.
\end{proof}
Fix a point $\ast\in\partial\D^n\subset\D^n$. 
Let $\Top(\D^n,\ast)$ denote the (closed) subgroup of $\Top(\D^n)$
consisting of all homeomorphisms $f:\D^n\to\D^n$ such that
$f(\ast)=\ast$. Note that $\On(n-1)$, identified with the subgroup of $\On(n)$ that fixes $\ast$,
acts naturally (on the right) on
$\Top(\D^n,\ast)$
\begin{lemma}
The orbit spaces $\TO{\D^n}{n}$ and $\TO{\D^n,\ast}{n-1}$ are
homeomorphic.
\end{lemma}
\begin{proof}
The subgroup inclusion $\Top(\D^n,\ast)\subset\Top(\D^n)$
induces a continuous injective map 
$\TO{\D^n,\ast}{n-1}\xrightarrow{\iota}\TO{\D^n}{n}$. Note that
for each $f\in\Top(\D^n)$, there exists an orthogonal transformation
$A\in\On(n)$ such that $A(f(\ast))=\ast$. It is easy to see that
the assignment $f\mapsto A\circ f$ induces an inverse
to $\iota$. Thus $\iota$ is a continuous bijection. But since
the orbit map $\Top(\D^n)\to\TO{\D^n}{n}$ is a closed map, $\iota$ is a 
closed map as well. Hence it is a homeomorphism.
\end{proof}
\begin{proof}[\textbf{Proof of Lemma \ref{Lemma 2}}:]
Since the restriction map $\Top(\D^n,\ast)\to\Top(S^{n-1},\ast)$ is a homotopy equivalence
(coning provides a homotopy inverse, by the Alexander trick), the induced map 
$\TO{\D^n,\ast}{n-1}\to\frac{\Top(S^{n-1},\ast)}{\On(n-1)}$ is a weak homotopy equivalence, which can be
seen by applying the Five Lemma to long exact sequences in homotopy associated to the fibrations
$\Top(\D^n,\ast)\to\TO{\D^n,\ast}{n-1}$ and $\Top(S^{n-1},\ast)\to\frac{\Top(S^{n-1},\ast)}{\On(n)}$.
Thus, as $\TO{\D^n}{n}$ is homeomorphic to $\TO{\D^n,\ast}{n-1}$,
it is weakly homotopy equivalent to $\frac{\Top(S^{n-1},\ast)}{\On(n-1)}\simeq\TO{n-1}{n-1}$, where the last 
identification is done via the stereographic projection.
It then follows from \cite[Essay V \S 5]{KS} and \cite{kervaire-milnor}, that  $\TO{\D^n}{n}$ is
simply connected and 
$\pi_j\TO{\D^n}{n}\otimes\Q=0$ for
all $1\leq j\leq n$. 
Now, by Morlet's comparison theorem we have that for all $j\geq n-1$
\beq
\pi_{j+1}\TO{n-1}{n-1}\simeq\pi_{j-(n-1)}\Diff(\D^{n-1},\partial).
\eeq
But Farrell and Hsiang \cite{FH78} showed that when $n$ is 
odd
$\pi_{j-(n-1)}\Diff(\D^{n-1},\partial)\otimes\Q=0$ if 
$0<j-(n-1)<\min\{\frac{n-8}{2},\frac{n-5}{3}\}$. Hence $\pi_i\Top(D^n)/\On(n)\otimes\Q=0$ for all $1\leq i<n+\min\{\frac{n-8}{2},\frac{n-5}{3}\}$. The result
now follows from Lemma \ref{general}.
\end{proof}
\section{Proof of Lemma \ref{Lemma 1}}\label{proofLemma1}
Let $(M,g)$ be a closed Riemannian manifold without conjugate points. Its 
tangent bundle $TM\xrightarrow{\pi} M$ is regarded as a Euclidean vector bundle
with respect to the metric $g$ (see \cite[p.21-22]{milnor-stasheff}).

The
tangent disc bundle $\D M$ is the collection of all vectors in $TM$ 
of length less than or equal to $1$. Let $\mbox{int}\D M\subset\D M$ be 
the subbundle of tangent vectors of length strictly less than $1$.

Let $\overline{\textbf{R}}_{\D}^T(M,g)$ denote the simplicial set
whose $k$-simplices are topological disc bundle isomorphisms 
\beq
\bar{f}:\Delta^k\times\D M
\to\Delta^k\times\D M
\eeq
fixing the zero section, which commute with the projection to $\Delta^k$ and cover
some homeomorphism $f:\Delta^k\times M\to\Delta^k\times M$.

A simplicial subset
$\overline{\textbf{R}}_{\D}^o(M,g)\subset\overline{\textbf{R}}_{\D}^T(M,g)$
is defined by requiring that
\beq
\bar{f}|_{\Delta^k\times\mbox{int}(\D M)}:
\Delta^k\times \D M\to
\Delta^k\times \D M
\eeq
is a fiberwise isometry.

The simplicial sets $\underline{\textbf{R}}_{\D}^o(M,g)$ and
$\underline{\textbf{R}}_{\D}^T(M,g)$ are defined similarly but they must
cover the identity map on $\Delta^k\times M$.
Note that the natural inclusion between simplicial sets
\beq
\sigma_{\D}:
\underline{\textbf{R}}_{\D}^T(M,g)/\underline{\textbf{R}}_{\D}^o(M,g)\to
\overline{\textbf{R}}_{\D}^T(M,g)/\overline{\textbf{R}}_{\D}^o(M,g)
\eeq
is a homotopy equivalence, by the same
argument given in \cite[p. 12]{BLTrans}. 

The map  
$S:\underline{\textbf{R}}^T(M)/\underline{\textbf{R}}^o(M)
\to\textbf{S}\Gamma(\calB_n(M))$ from Section \ref{Morlet} can
easily be adapted (see Remark \ref{rmk3} in Appendix) to obtain a corresponding map 
\beq
S_{\D}:
\underline{\textbf{R}}_{\D}^T(M,g)/\underline{\textbf{R}}_{\D}^o(M,g)
\to\textbf{S}\Gamma(\calB^{\D}_{n}(M)).
\eeq

Now fix the fiber preserving homeomorphism
$\phi:\mbox{int}(\D M)\to TM$ defined by 
\beq
\phi(v)=\frac{v}{1-\sqrt{\langle v,v\rangle_{\pi(v)}}}.
\eeq
We can then form a commutative diagram
\beq
\xymatrix{
\overline{\textbf{R}}_{\D}^T(M,g)/\overline{\textbf{R}}_{\D}^o(M,g)\ar[d]^{\mbox{int}_{\phi}} &  \underline{\textbf{R}}_{\D}^T(M,g)/\underline{\textbf{R}}_{\D}^o(M,g)\ar[l]_{\sigma_{\D}}\ar[d]^{\mbox{int}_{\phi}}\ar[r]^-{S_{\D}} & \textbf{S}\Gamma(\calB_{n}^{\D}(M))\ar[d]^{\mbox{int}_{\varphi}}\\
\overline{\textbf{R}}^t(M)/\overline{\textbf{R}}^d(M) & \underline{\textbf{R}}^T(M)/\underline{\textbf{R}}^o(M)
\ar[l]_{\sigma}\ar[r]^-{S} & 
\textbf{S}\Gamma(\calB_{n}(M))
}
\eeq
where the vertical maps  are induced by restriction of a 
homeomorphism of $\D^n$ (or $\D M$) to its interior and then 
conjugation by $\varphi$ (or $\phi$) and $\sigma$
is the homotopy equivalence (\ref{other maps}) of the Appendix below. 
Taking homotopy groups yields the following commutative diagram

\begin{equation}\label{repsD}
\xymatrixcolsep{5pc}\xymatrix{
\pi_{i+1}\overline{\textbf{R}}^T_{\D}(M,g)/\overline{\textbf{R}}^o_{\D}(M,g)\ar[r]^-{S_{\D\ast}\circ\sigma_{\D\ast}^{-1}}\ar[d]_{\mbox{int}_{\phi\ast}} & 
\pi_{i+1}\textbf{S}\Gamma(\calB_n^{\D}(M))\ar[d]^{\mbox{int}_{\varphi}\ast}\\
\pi_{i+1}\overline{\textbf{R}}^t(M)/\overline{\textbf{R}}^d(M)
\ar[r]^-{S_*\circ\sigma_*^{-1}} &
\pi_{i+1}\textbf{S}\Gamma(\calB_n(M))
}
\end{equation}
\begin{proof}[\textbf{Proof of Lemma \ref{Lemma 1}}]
Let $f:\Delta^{i+1}\times M\to\Delta^{i+1}\times M$ be 
an $(i+1)$-simplex of $\sTop_0(M)$ representing an element $[f]$ in
$\pi_{i+1}(\sTop_0(M),\sDiff_0(M))\simeq\pi_{i+1}\frac{\Top_0(M)}{\Diff_0(M)}$,
such that $\Phi^g_{\ast}\circ d_*[f]=0$, i.e. 
$f_z:=f|_{\{z\}\times M}\in\Diff_0(M)$
for all $z\in\partial\Delta^{i+1}$ and there exists a continuous map $\Delta^{i+1}\to\metnc(M)$,
$z\mapsto g_z$, 
with the property that $g_z=(f_z)_*g$ for all $z\in \partial\Delta^{i+1}$.
We shall exhibit an element in 
$\pi_{i+1}\Gamma(\calB^{\D}_n(M))$ whose image under
$\mbox{int}_{\varphi*}$ coincides with
$\mu_*[f]$ (refer to the notation in Section \ref{Morlet}).

Fix a universal cover $\widetilde{M}$ of $M$ and let $\Gamma$ be its group of deck transformations.
The homeomorphism $\Delta^{i+1}\times M\to\Delta^{i+1}\times M$ sending
$(z,x)$ to $(z,f_z(x))$ lifts to a homeomorphism 
$\Delta^{i+1}\times\widetilde{M}\to\Delta^{i+1}\times\widetilde{M}$ which in turn
restricts to a homeomorphism 
$\widetilde{f}_z:\widetilde{M}\to\widetilde{M}$ for each $z\in\Delta^{i+1}$.

The universal cover
$\widetilde{M}$ acquires an $(i+1)$-parameter family
of complete Riemannian metrics 
$\widetilde{g}_z=p^*g_z$ without conjugate points. Note that for all $z\in\Delta^{i+1}$, 
there exists a canonical vector bundle isomorphism (over the identity)
$\rho_{g_z}:TM\to TM$ such that $\rho_{g_z}$ is  a fiberwise isometry between the fixed Euclidean
structure on $TM$ (i.e. $g$) and the one given by the metric $g_z$ (see \cite[p.24]{milnor-stasheff}). This can
be thought of as an $(i+1)$-simplex  $\mathfrak{R}\in\underline{\textbf{R}}^d(M)$
defined by $\mathfrak{R}(z,v)=(z,\rho_{g_z}(v))$, for all $(z,v)\in\Delta^{i+1}\times TM$. 

Let $\exp^{\tilde{g}}:T\widetilde{M}\to\widetilde{M}\times\widetilde{M}$ be the exponential map with respect to
the metric $\tilde{g}$, i.e. $\exp^{\tilde{g}}(v)=(\mbox{foot}(v),\gamma_v(1))$, where $\gamma_v$ is the
unique geodesic with initial velocity $v$. The corresponding map mod $\Gamma$ is denoted by
$\exp^g_{\Gamma}:TM\to\widetilde{M}\times_{\Gamma}\widetilde{M}$.

Now for each $z\in\Delta^{i+1}$ define the following
map:
\beq
\Upsilon_z:TM\xrightarrow{\exp_{\Gamma}^{g}}
\widetilde{M}\times_{\Gamma}\widetilde{M}
\xrightarrow{\widetilde{f_z}\times_{\Gamma}\widetilde{f_z}}
\widetilde{M}\times_{\Gamma}\widetilde{M}
\xrightarrow{(\exp_{\Gamma}^{g_z}\circ\rho_{g_z})^{-1}}
TM
\eeq
Because the exponential map of a Riemannian metric on $\widetilde{M}$ without conjugate points is a diffeomorphism, these maps are $\R^n$-bundle isomorphisms  which cover $f_z$.
Moreover, when $z\in \partial\Delta^{i+1}$, the map $\Upsilon_z$ is 
a linear isomorphism.
Thus we obtain
a class $[\Upsilon]$ in the relative
homotopy group 
$\pi_{i+1}\left(\overline{\textbf{R}}^t(M),\overline{\textbf{R}}^d(M)\right)$. Furthermore,
by the definition of $\delta$ (see Section \ref{Morlet}), we
have the equation
\begin{equation}\label{equation}
\delta_*[f]=[\mathfrak{R}\circ\Upsilon]=[\mathfrak{R}]+[\Upsilon]=[\Upsilon]
\in\pi_{i+1}\left(\overline{\textbf{R}}^t(M),\overline{\textbf{R}}^d(M)\right)
\simeq\pi_{i+1}\overline{\textbf{R}}^t(M)/\overline{\textbf{R}}^d(M),
\end{equation}
since $[\mathfrak{R}]=0$.

Now we exploit the geometry at infinity of the
universal cover of $M$ to construct topological representations
of the disc bundle of $M$.

Observe that since $M$ is compact, the metrics $\widetilde{g_z}$ are all
quasi-isometric to $\widetilde{g}=p^*g$, and since the fundamental
group of $M$ is hyperbolic we can extend each 
self-homeomorphism
$\widetilde{f_z}$ to a self-homeomorphism 
$\overline{f_z}:\overline{M}\to\overline{M}$ of 
the geometric compactification $\overline{M}$ of $\widetilde{M}$,
whose points at infinity are represented by $g$-quasi-geodesic rays in
$\widetilde{M}$ (see \cite[p. 400-405]{bridson}).

This yields a homeomorphism
\beq
\overline{\Upsilon}_z:\D M\xrightarrow{\Exp_{\Gamma}^{g}}
\widetilde{M}\times_{\Gamma}\overline{M}
\xrightarrow{\widetilde{f_z}\times_{\Gamma}\overline{f_z}}
\widetilde{M}\times_{\Gamma}\overline{M}
\xrightarrow{(\Exp_{\Gamma}^{g_z}\circ\rho_{g_z})^{-1}}
\D M
\eeq
which clearly covers $f_z$, and is a fiberwise isometry for $z\in\partial\Delta^{i+1}$.
Here $\Exp^g_{\Gamma}$ 
is given by the following map $\mbox{mod }\Gamma$
\beq
v\mapsto
\begin{cases}
\left(\pi(v),\exp^g_{\pi(v)}(v/1-||v||_g)\right) & \text{ if } ||v||_g<1\\
\left(\pi(v),\gamma_v(\infty)\right) & \text{ if } ||v||_g=1,
\end{cases}
\eeq
where $\gamma_v(\infty)$ is the point at infinity 
determined by the geodesic ray $\widetilde{M}$ with initial
velocity $v$.

Hence the collection of maps $\overline{\Upsilon}_z$, 
$z\in\Delta^{i+1}$,
determines a class
$[\overline{\Upsilon}]$
in $\pi_{i+1}\left(\overline{\textbf{R}}_{\D}^T(M,g),\overline{\textbf{R}}_{\D}^o(M,g)\right)
\simeq\pi_{i+1}\left(\overline{\textbf{R}}_{\D}^T(M,g)/\overline{\textbf{R}}_{\D}^o(M,g)\right)$ with the
property that
\beq
\mbox{int}_{\phi*}[\overline{\Upsilon}]
=[\Upsilon],
\eeq
The result now follows from equation (\ref{equation}) and
the commutativity of diagram (\ref{repsD}). 
\end{proof}
\section*{Appendix}\label{appendix}
Here we prove Lemma \ref{R^t} and Lemma \ref{the map S}, following ideas from \cite{BLTrans}.  

\begin{proof}[\textbf{Proof of Lemma \ref{R^t}}]
Consider the following composition
\begin{equation}\label{other maps}
\sigma:
\underline{\textbf{R}}^T(M)/\underline{\textbf{R}}^o(M)\rightarrow
\underline{\textbf{R}}^t(M)/\underline{\textbf{R}}^o(M)\rightarrow
\underline{\textbf{R}}^t(M)/\underline{\textbf{R}}^d(M)\rightarrow
\overline{\textbf{R}}^t(M)/\overline{\textbf{R}}^d(M)
\end{equation}
induced by the obvious maps.
We will prove that $\sigma$
is a homotopy equivalence. The rightmost map is a weak homotopy equivalence by
\cite[Proposition 1.4]{BLTrans}). The leftmost map is a 
weak equivalence by Kister-Mazur theorem. In fact, consider the fiber bundle 
$\calE(M)\to M$ whose fiber over $x\in M$ is the space 
$\mathcal{EMB}(\mathcal{U}_0,T_xM)$ of germs
of embeddings of $\mathcal{U}_0$ into $T_xM$ fixing the origin, for some neighbourhood 
$\mathcal{U}_0$ of the origin in $T_xM$.
Similarly we define the bundle $\calT(M)\to M$ with fiber the
space $\Top(T_xM,0)$ of self-homeomorphisms of $T_xM$ fixing the origin. 
There are obvious simplicial isomorphisms 
$\underline{\textbf{R}}^T(M)\to\textbf{S}\Gamma(\calT(M))$
and
$\underline{\textbf{R}}^t(M)\to\textbf{S}\Gamma(\calE(M))$ which make the following 
diagram commutative
\beq
\xymatrix{
\underline{\textbf{R}}^T(M)\ar[d]\ar[r]&\textbf{S}\Gamma(\calT(M))\ar[d]\\
\underline{\textbf{R}}^t(M)\ar[r]&\textbf{S}\Gamma(\calE(M)).
}
\eeq

It follows from Kister's theorem that the right vertical map is a weak homotopy
equivalence. To see this, note that the restriction of $\calT(M)\to\mathcal{E}(M)$ 
to an individual fiber is the map $\Top(\R^n,0)\to\mathcal{EMB}(\mathcal{U}_0,\R^n)$. This map factors as the
composite
$\Top(\R^n,0)\to\mbox{Emb}(\R^n,0)\to\mathcal{EMB}(\mathcal{U}_0,\R^n)$ 
where the middle term is the space of continuous
self-embeddings of $\R^n$ fixing the origin. Kister \cite{kister} showed that the first map is a homotopy equivalence. The second map
is also a homotopy equivalence as can be seen by first passing to the corresponding simplicial fibration between their 
singular complexes, and then changing PL-equivalences by continuous embeddings in  
\cite[Lemma 1.5. Part c) and e)]{lashof-kuiper} to show that the fiber is contractible.

An obstruction theory argument now shows that the map between section spaces is
also a weak homotopy equivalence (see for example \cite[Lemma 3.10]{BLAnnals}) . Therefore the map on the left is 
a weak homotopy equivalence.

It is now easy to see that the vertical left map induces the
desired weak homotopy equivalence
\beq
\underline{\textbf{R}}^T(M)/\underline{\textbf{R}}^o(M)
\to
\underline{\textbf{R}}^t(M)/\underline{\textbf{R}}^o(M)
\eeq

Finally, note that since $\On(n)$ is homotopy equivalent to $GL_n(\R)$, it also follows that 
$\underline{\textbf{R}}^o(M)\subset\underline{\textbf{R}}^d(M)$ is a weak homotopy equivalence and hence
so is the middle map in (\ref{other maps}). This completes the proof.
\end{proof}

We now prove Lemma \ref{the map S}.

Let $G$ be a topological group and $H$ a closed subgroup of $G$ such that
$G\to G/H$ has local cross sections. For example $G=\Top(\R^n,0)$ or $\Top(\D^n,0)$ and
$H=\On(n)$.

Let $\calB_0: X\to B$ be a fiber bundle with fiber $Y$ and structure group $H$. Suppose that action of
$H$ on $Y$ extends to an action of $G$ on $Y$,  so that we can also
we regard $\calB_0$ as a bundle with structure group $G$.

Let $\underline{\textbf{R}}^G$ be the simplicial set with $k$-simplices given by
commutative diagrams
\beq
\xymatrix
{&\Delta^k\times X\ar[rr]^-{F}\ar[rd] & &\Delta^k\times X\ar[ld]\\
&& \Delta^k\times B
}
\eeq
where $F$ is a $G$-equivalecnce, i.e. an isomorphism between bundles with structure group $G$.

Likewise one defines the simplicial set $\underline{\textbf{R}}^H$, and observe that there is an
inclusion $\underline{\textbf{R}}^H\subset\underline{\textbf{R}}^G$. Following Burghelea-Lashof
\cite[p.29]{BLTrans} and Steenrod \cite[p. 44-45]{Steenrod}, we define a simplicial map
\beq
\widehat{S}:\underline{\textbf{R}}^G/\underline{\textbf{R}}^H\to
\textbf{S}\Gamma\left(P_HX\times_HG/H\right),
\eeq
where $P_HX\to B$ is the principal $H$-bundle associated to $\calB_0$. The map is defined as follows:
let $\{V_j,\phi_j\}$ be a local coordinate system for $\calB_0$ and let $\{V_j,\phi_j'\}$ be the
corresponding local coordinate system for $P_HX\times_HG/H\to B$.

Given a $k$-simplex $F:\Delta^k\times X\to\Delta^k\times X$ of $\underline{\textbf{R}}^G$,  for each
$t\in\Delta^k$ we obtain
a map 
\beq
\phi_j^{-1}\circ F_t\circ\phi_j: V_j\times Y\to V_j\times Y
\eeq
which preserves the first component. Hence it is determined by some map
\beq
\lambda_j^t:V_j\to G,
\eeq
such that $\phi_j^{-1}\circ F_t\circ\phi_j(x,y)=(x,\lambda_j^t(x)y)$.

Let now $\rho:G\to G/H$ be the quotient map. Then the map
\beq
s:\Delta^k\times B\to \Delta^k\times P_HX\times_H\times G/H
\eeq
defined by $s(t,x)=\phi'_j(x,\rho(\lambda_j^t(x)))$, $x\in V_j$ is a section of the bundle 
$\Delta^k\times (P_HX\times_H\times G/H)\to\Delta^k\times B$, and hence a $k$-simplex of 
$\textbf{S}\Gamma\left(P_HX\times_HG/H\right)$. This gives a map 
$\underline{\textbf{R}}^G\to\textbf{S}\Gamma\left(P_HX\times_HG/H\right)$ which is independent of
choices of $\{V_j,\phi_j\}$ and
which induces the desired map
\beq
\widehat{S}:\underline{\textbf{R}}^G/\underline{\textbf{R}}^H
\to
\textbf{S}\Gamma\left(P_HX\times_HG/H\right).
\eeq
\begin{propo}\label{proposition}
The map $\widehat{S}$ sends the $0$-simplicies of $\underline{\textbf{R}}^G/\underline{\textbf{R}}^H$
into a union, say $\Gamma_0\left(P_HX\times_HG/H\right)$, of path components of 
$\Gamma\left(P_HX\times_HG/H\right)$. Furthermore 
$\widehat{S}:\underline{\textbf{R}}^G/\underline{\textbf{R}}^H
\to
\textbf{S}\Gamma_0\left(P_HX\times_HG/H\right)$ is a simplicial isomorphism.
\end{propo}

Assuming this proposition, the proof of Lemma \ref{the map S} is straightforward:
\begin{proof}[\textbf{Proof of Lemma \ref{the map S}}]
Set $G=\Top(\R^n,0)$ and $H=\On(n)$ in the previous proposition. Then compose $\widehat{S}$ with the map
$\Gamma(P_H\times_HG/H)\to\Gamma(\calB_n(M))$, induced by the inclusion $G\to\Top(\R^n)=\Top(n)$.
\end{proof}

\begin{rmk}\label{rmk3}
Now set $G=\Top(\D^n,0)$ and $H=\On(n)$ in the previous proposition and compose $\widehat{S}$ with the map
$\Gamma(P_H\times_HG/H)\to\Gamma(\calB_n^{\D}(M))$, induced by the inclusion $G\to\Top(\D^n)$.
This gives the
map $S_{\D}$ from Section \ref{proofLemma1}.
\end{rmk}

In order to prove Proposition \ref{proposition} we restate the theorem in
\cite[p. 45]{Steenrod}. First regard $\calB_0$ as a $Y$-bundle $\calB: X\xrightarrow{\pi} B$
with structure group $G$. An $(H,Y)$-bundle structure on $\calB$ is an equivalence class of
pairs $[(\calB_1,F_1)]$
\begin{itemize}
\item $\calB_1$ is a bundle $X_1\xrightarrow{\pi_1} B$ with fiber $Y$ and structure group $H$.
\item $F_1:X_1\to X$ is a $G$-equivalence over $B$.
\item Two pairs $(\calB_1,F_1)$ and $(\calB_2,F_2)$ are equivalent if and only if there exists
an $H$-equivalence $\calF:X_1\to X_2$ over $B$ such that $F_2\circ\calF = F_1$.
\end{itemize}

Let $\calS_{(H,Y)}(\calB)$ denote the set of equivalence classes of $(H,Y)$-bundle structures on
$\calB$. Denote by $\pi':P_HX\times_{H}G/H\to B$ the bundle projection. For each equivalence class
$[(\calB_1,F_1)]$, note that $F_1:X_1\to X$ 
determines a collection of maps $\{\lambda_j:V_j\to G\}$. Together, they determine a map
$B\to P_HX\times_HG/H$
\beq
x\mapsto\phi_j'(x,\rho(\lambda_j(x))),
\eeq
for $x\in V_j$, which is a section
of the bundle $P_HX\times_H G/H$. Thus we have a map 
\beq
\calS:\calS_{(H,Y)}(\calB)\to\Gamma(P_HX\times_HG/H).
\eeq

Let $q: EH\to BH$ be the universal $H$-bundle and consider the $H$-bundle map
$\bar{h}_0:P_HX\to EH$ covering a map $h_0:B\to BH$ which classifies the bundle 
$P_HX\to B$. Identify $EH$ and $EG$ so that $G$ acts on $EH$ on the right with orbit space
$BG$ and let
$\hat{h}_0:P_HX\times_H G/H\to BH$ be given by $[x,\rho(\alpha)]\mapsto q(\bar{h}_0(x)\cdot\alpha)$,
for $(x,\alpha)\in P_HX\times G$.

\begin{rmk}
$\hat{h}_0:P_HX\times_H G/H\to BH$ is a $G/H$-bundle map over $i\circ h_0:B\to BG$
where $i:BH\to BG$ is the quotient map.
\end{rmk}
\begin{lemma}\label{lemma}
The following holds:
\begin{enumerate}
\item $\calS$ is a bijection.
\item $B\xrightarrow{\calS([\calB_1,F_1])} P_H\times_H G/H\xrightarrow{\hat{h}_0} BH$ is a classifying map for $\calB_1$.
\end{enumerate}
\end{lemma}

We now proceed to prove Prosposition \ref{proposition} assuming Lemma \ref{lemma}.

\begin{proof}[\textbf{Proof of Prosposition \ref{proposition}}]
For any 0-simplex $F:X\to X$ of $\underline{\textbf{R}}^G$ one has 
$\widehat{S}(F)=\calS[(\calB_0,F)]$. So by the first part of Lemma
\ref{lemma}, the map $\widehat{S}$ is monic on the $0$-simplices.
 Thus to show the first part of Proposition \ref{proposition}
it suffices to show that if $s_t$, $t\in [0,1]$ is a path in $\Gamma(P_HX\times_{H}G/H)$ such that
$s_0\in\mbox{im}(\widehat{S})$, then $s_1\in\mbox{im}(\widehat{S})$. 
In fact, suppose
that $\widehat{S}(F)=s_0$ by Lemma \ref{lemma} there exists
$[(\calB_1,F_1)]$ such that $\calS([\calB_1,F_1])=s_1$. 
By Part 2. of Lemma \ref{lemma}, since $s_1$ is path connected to $s_0$, there exists an
$H$-equivalence $h:\calB_1\to\calB_0$ (because their classifying maps are homotopic).
Hence $\widehat{S}(F_1\circ h^{-1})=\calS([\calB_0, F_1\circ h^{-1}])=\calS[\calB_1,F_1]=s_1$.

This shows the first part of the proposition. The second part follows from a similar argument.
\end{proof}

We now prove the Lemma \ref{lemma}.
\begin{proof}[\textbf{Proof of Lemma \ref{lemma}}]
The first part is the same as \cite[9.4, p. 44-45]{Steenrod}.
To show the second part, it suffices to find an $H$-bundle map $P_HX_1\to EH$ covering
$B\xrightarrow{\calS[\calB_1,F_1]}P_HX\times_H G/H\xrightarrow{\hat{h}_0} BH$,
where $P_HX_1$ is the associated principal $H$-bundle of $\calB_1: X_1\to B$.

Since there is a natural $H$-bundle map $P_GX=P_HX\times_HG\to EH$ given by
$[x,\alpha]\mapsto\bar{h}_0(x)\cdot\alpha$, covering $\hat{h}_0:P_HX\times_HG/H\to BH$; we 
only need to give an $H$-bundle map $P_HX_1\to P_GX$ covering the map
$B\xrightarrow{\calS[\calB_1,F_1]} P_HX\times_H G/H$.

This map is defined via the following commutative diagram
\beq
\xymatrix{
(\pi_1^H)^{-1}(V_j)\ar[r] & (\pi^G)^{-1}(V_j)\\
V_j\times H\ar[r]\ar[u]_{\varphi_j^H}   & V_j\times G\ar[u]_{\phi_j^G}
}
\eeq
where  $\{V_j,\phi_j^G\}$ is a local coordinate system for the principal $G$-bundle $\pi^G: P_GX\to B$, and $\{V_j,\phi_j^H\}$ is a local coordinate
system for the principal $H$-bundle $P_HX_1\xrightarrow{\pi_1^H} B$.
The lower horizontal map is given by $(x,h)\mapsto (x,\lambda_j(x)h)$ and
$\{\lambda_j:V_j\to G\}$ is induced by $F_1$. 
\end{proof}


\bibliographystyle{alpha} 
 {\footnotesize\bibliography{biblio}} 
\Addresses

\end{document}